\documentclass{amsart}

\usepackage[latin1]{inputenc}

\usepackage{amsmath,amssymb,amscd,latexsym,color,ifthen,enumerate}

\usepackage[dvips]{graphicx}

\usepackage[english]{babel}
\usepackage{verbatim}

\newcommand{\communique}{\leftrightarrow}

\newcommand{\N}{\mathbb{Z}_{+}}

\newcommand{\Zd}{\mathbb{Z}^d}
\newcommand{\C}{\mathbb{C}}

\newcommand{\R}{\mathbb{R}}
\newcommand{\Rd}{\mathbb{R}^d}

\renewcommand{\P}{\mathbb{P}}
\newcommand{\E}{\mathbb{E}}

\newcommand{\Ed}{\mathbb{E}^d}

\renewcommand{\epsilon}{\varepsilon}
\renewcommand{\phi}{\varphi}
\renewcommand{\limsup}{\overline{\lim}}

\newcommand{\ie}{\emph{i.e. }}

\newcommand{\miniop}[3]{%
\renewcommand{\arraystretch}{0.6}
\begin{array}{c}
{\scriptstyle #1}\\
#2\\
{\scriptstyle #3}
\end{array}
\renewcommand{\arraystretch}{1}}

\newcommand{\1}{1\hspace{-1.3mm}1}
\newcommand{\Var}{\text{Var }}

\newcommand{\A}[1]{A_{\ref{#1}}}
\newcommand{\alphaa}[1]{\alpha_{\ref{#1}}}
\newcommand{\B}[1]{B_{\ref{#1}}}
\newcommand{\betaa}[1]{\beta_{\ref{#1}}}
\renewcommand{\C}[1]{C_{\ref{#1}}}

\newcommand{\D}[1]{D_{\ref{#1}}}

\begin{document}

\newtheorem{theorem}{Theorem}[section]
\newtheorem{conjecture}[theorem]{Conjecture}
\newcommand{\titre}{\title}

\newtheorem{lemme}[theorem]{Lemma}
\newtheorem{defi}[theorem]{Definition}
\newtheorem{coro}[theorem]{Corollary}
\newtheorem{rem}[theorem]{Remark}
\newtheorem{prop}[theorem]{Proposition}

\title[Moderate deviations for the chemical distance]{Moderate deviations for the chemical distance in Bernoulli percolation}

\author{Olivier Garet}
\address{Institut \'Elie Cartan Nancy (mathématiques)\\
Université Henri Poincaré Nancy 1\\
Campus Scientifique, BP 239 \\
54506 Vandoeuvre-lès-Nancy  Cedex France\\}
\email{Olivier.Garet@iecn.u-nancy.fr}
\author{Régine Marchand}
\email{Regine.Marchand@iecn.u-nancy.fr}

\address{Institut \'Elie Cartan Nancy (mathématiques)\\
Nancy-Université, CNRS,\\
Boulevard des Aiguillettes B.P. 239\\
F-54506 Vand\oe uvre-lès-Nancy,  France.}

\def\motsclefs{Percolation, chemical distance, moderate deviations, concentration inequalities, subadditivity.}

\subjclass[2000]{60K35, 82B43.} 
\keywords{\motsclefs}

\begin{abstract}
In this paper, we establish moderate deviations for the chemical distance in Bernoulli percolation. The chemical distance between two points is the length of the shortest open path between these two points. Thus, we study the size of random fluctuations around the mean value, and also the asymptotic behavior of this mean value. The estimates we obtain improve our knowledge of the convergence to the asymptotic shape. Our proofs rely on concentration inequalities proved by Boucheron, Lugosi and Massart, and also on the approximation theory of subadditive functions initiated by Alexander.
\end{abstract}
\thanks{Research supported by Agence Nationale de la Recherche (France),  ANR-08-BLAN-0190.}

\maketitle

\section{Introduction and results}

We consider supercritical Bernoulli percolation on the edges of $\Zd$, where $d\ge 2$ is a fixed integer. The set of edges of $\Zd$ is denoted by $\Ed$ and the set $\Omega=\{0,1\}^{\Ed}$ is endowed with the probability $\P=\mathcal{B}(p)^{\otimes \Ed}$: the coordinates $(\omega_e)_{e \in \Ed}$ are thus independent and identically distributed random variables following the Bernoulli law with parameter $p$. We denote by $p_c(\Zd)$  the critical point for Bernoulli percolation on $\Zd$, and in the following, $p>p_c(\Zd)$ is  fixed.

We denote by $C_{\infty}$ the infinite percolation cluster. The chemical distance between two points $x, y \in \Zd$ is the length of the shortest open path between these two points. Asymptotically, this chemical distance is equivalent to a deterministic norm $\mu$ -- 
see  Garet and Marchand~\cite{GM-fpppc}. Moreover, we even obtain large deviations inequalities in~\cite{GM-large}: 
\begin{equation*}
\label{GDmu}
\forall \epsilon>0 \quad \miniop{}{\limsup}{\Vert y\Vert_1\to +\infty} \frac{\log \P \left(0\communique   y, \; \frac{D(0,y)}{\mu(y)}\notin (1-\epsilon, 1+\epsilon)\right)}{\Vert y\Vert_1}<0.
\end{equation*}
Our main results here are the following moderate deviation estimates:
\begin{theorem}
\label{concentrationD}
There exists a constant $\C{concenD}>0$ such that
\begin{equation}
\forall y \in \Zd \quad \E \left( \left| D(0,y)-\mu(y)\right| \1_{\{0\communique   y\}} \right) \le \C{concenD}\sqrt{\|y\|_1}\log(1+\|y\|_1). \label{concenD}
\end{equation}
There exist constants $\A{equmoderter},\B{equmoderter},\C{equmoderter}>0$  such that \\
$\forall y\in \Zd \backslash \{0\}$, $\forall x\in [\C{equmoderter} (1+\log \|y\|_1),\|y\|_1^{1/2}]$
\begin{eqnarray}
&& \quad  \quad \quad \P \left( 
\frac{| D(0,y)-\mu(y)|}{\sqrt{\|y\|_1}}> x , \;
 0 \communique y
\right) \le \A{equmoderter} e^{- \B{equmoderter} x}. \label{equmoderter}
\end{eqnarray}
There exists a constant $\C{equtfa}>0$ such that, $\P$-almost surely, on the event $\{0 \communique \infty\}$, for  $t$ large enough,
\begin{equation}
\mathcal{B}_{\mu}^0(t-\C{equtfa}\sqrt{t}\log t)\cap C_{\infty} \subset B^0(t)\subset  \mathcal{B}_{\mu}^0(t+\C{equtfa}\sqrt{t}\log t), \label{equtfa}
\end{equation}
where $B^0(t)=\{x \in\Zd: \;  D(0,x)\le t\}$ is the ball with radius $t$ for the chemical distance and  $\mathcal{B}_{\mu}^0(t)$ is the ball with radius $t$ for the norm $\mu$ .
\end{theorem}
In Kesten's work on first-passage percolation, the analogue of Inequality~(\ref{equmoderter}) is valid for $x \in[0, \|y\|_1]$.
Note however that for $x\in[\sqrt{\|y\|_1},\|y\|_1]$, the large deviation estimates are more accurate; on the other hand, we cannot extend our results for $x\in[0,\C{equmoderter} (1+\log \|y\|_1)]$ because of the approximation and renormalization process we use in the proof. 

Let us note first that the chemical distance between two points is infinite as soon as these points are not in the same open cluster: to overcome this problem, we introduce a variation of the chemical distance in the following way: for  $x\in \Zd$, we denote by $x^*$ the closest point to $x$ in  $C_\infty$ (for the $\|.\|_1$-distance). For indeterminate cases, we choose $x^*$ to minimize $x^*-x$ for a given deterministic rule, as the lexicographic order for instance. Then, we define:
$$\forall x,y \in \Zd \quad D^*(x,y)=D(x^*,y^*).$$
We will mainly work with $D^*$ and we will come back to the chemical distance $D$ by bounding the discrepancy between the two quantities. Here are the results we obtain for $D^*$:
\begin{theorem}
\label{concentrationD*}
There exists a constant  $\C{equvar}>0$ such that
\begin{equation}
\forall y \in \Zd \quad \Var D^*(0,y)\le \C{equvar}\|y\|_1\log(1+\|y\|_1). \label{equvar}
\end{equation}
For every $\D{equmoder}>0$, there exist constants $\A{equmoder},\B{equmoder},\C{equmoder}>0$ such that
for each $y \in\Zd \backslash \{0\}$,  for each $x\in [\C{equmoder} (1+\log \|y\|_1),\D{equmoder}\sqrt{\|y\|_1}]$,
\begin{equation}
\P \left( \frac{| D^*(0,y)-\E[D^*(0,y)] |}{\sqrt{\|y\|_1}}> x \right) \le \A{equmoder} e^{- \B{equmoder} x}. \label{equmoder}
\end{equation}
There exists a constant $\C{equetlesperance1}>0$  such that
\begin{equation}
\forall y \in \Zd\backslash\{0\} \quad 0  \le  \E[D^*(0,y)] -\mu(y)  \le  \C{equetlesperance1} \sqrt{\|y\|_1}\log (1+\|y\|_1). \label{equetlesperance1}
\end{equation}
The last two inequalities easily give some constants $\A{equmoderbis},\B{equmoderbis},\C{equmoderbis}>0$ such that for every $y \in \Zd \backslash \{0\}$,
\begin{equation}
\forall x\in [\C{equmoderbis} (1+\log \|y\|_1),\sqrt{\|y\|_1}]\quad  \P \left( \frac{| D^*(0,y)-\mu(y)|}{\sqrt{\|y\|_1}}> x \right) \le \A{equmoderbis} e^{- \B{equmoderbis} x}. \label{equmoderbis}
\end{equation}
\end{theorem}
As in first-passage percolation (see Kesten~\cite{kesten-modere}), the proof naturally falls into two parts:
\begin{itemize}
\item the control of the random fluctuations of $D^*(0,x)$ around its mean value (concentration property),
\item the control of the discrepancy between the mean value of $D^*(0,x)$ and  $\mu(x)$.
\end{itemize}

In his original work, Kesten used martingale technics. Such technics were also used by Howard and Newman~\cite{MR1452554,MR1849171} for Euclidean first-passage percolation (see also the survey by Howard~\cite{MR2023652}) and by Pimentel~\cite{pimentel-preprint} for the Vahidi-Asl and Wierner model. 
Since then,  the  concentration method developed by Talagrand~\cite{MR1361756} offered a new approach to this kind of problems and allowed to improve Kesten's estimates. In the same spirit, Benaïm and Rossignol~\cite{MR2451057} managed to enhance variance estimates in first-passage percolation. Here, we bound the fluctuations~(\ref{equvar}) and~(\ref{equmoder}) thanks to recent concentration inequalities of Boucheron, Lugosi and Massart~\cite{MR1989444}, which are easier to use than the abstract result of Talagrand. 

The control~(\ref{equetlesperance1}) of the discrepancy between $\E[D^*(0,y)]$ and~$\mu(y)$ usually relies on the moderate deviations estimates 
for the fluctuations of $D^*(0,x)$ around its mean value; it is particularly clear for the models with spherical symmetry such as Euclidean first-passage percolation~\cite{MR1452554,MR1849171}. A symmetry argument can also give simple proofs in the main direction, see Alexander~\cite{MR1202516}. Since this argument does not apply in an arbitrary direction, we choose to use the technics developed by Alexander~\cite{Alex97} for the approximation of subadditive functions.

The paper is organized as follows: in Section 2, we give some estimates for the chemical distance $D$ and its variation $D^*$, that are mainly derived from Antal and Pisztora's results~\cite{AP96}; we also deduce Theorem~\ref{concentrationD} for $D$ from Theorem~\ref{concentrationD*} for $D^*$. In Section 3, we prove Inequalities~(\ref{equvar}) and~(\ref{equmoder}) of Theorem~\ref{concentrationD*}, using concentration inequalities of Boucheron, Lugosi and Massart on one hand, and a mesoscopic renormalization argument on the other hand. Section~4 is devoted to the control of the discrepancy~(\ref{equetlesperance1}) between $\mu(x)$ and $\E[D^*(0,x)]$, following Alexander's method.


\section{Some inequalities}


\subsection{Classical estimates}

For each  $x$ in $\Zd$, we denote by $C(x)$ the percolation cluster 
of $x$; $|C(x)|$ is its cardinal.
Thanks to Chayes, Chayes, Grimmett, Kesten and Schonmann~\cite{CCGKS}, we
can control the radius of finite clusters: there exist 
constants $\A{amasfini},\B{amasfini}>0$ such that
\begin{equation}
\label{amasfini}
\forall r>0 \quad
 \P \left( |C(0)|<+\infty, \; C(0) \not\subset [-r,\dots,r]^d \right)
\le \A{amasfini}e^{-\B{amasfini} r}.
\end{equation}

We can also control the size of the holes in the infinite cluster: there exist
two strictly positive constants   $\A{amasinfini}$ et $\B{amasinfini}$ 
such that
\begin{equation}
\label{amasinfini}
\forall r>0 \quad
 \P \left( C_\infty \cap  [-r,\dots,r]^d=\varnothing \right)
\le \A{amasinfini}e^{-\B{amasinfini} r}.
\end{equation}
When $d=2$, this result follows from the large deviation estimates by Durrett
and Schonmann~\cite{DS}. Their methods can easily be transposed when $d\ge 3$.
Nevertheless, when $d\ge 3$, the easiest way to obtain it seems to use Grimmett
and Marstrand~\cite{Grimmett-Marstrand} slab's result.

\subsection{Estimates for the chemical distance}

The following lemma is a consequence of an auxiliary result obtained by  Antal and Pisztora~\cite{AP96}. This lemma actually contains the two theorems stated in their article.
\begin{lemme}
There exist positive constants $\alphaa{APexp},\betaa{APexp}$ such that
\begin{equation}
\label{APexp}
\forall y \in \Zd \quad \E[e^{\alphaa{APexp} \1_{\{0\communique y\}}D(0,y)}]\le e^{\betaa{APexp} \|y\|_1}.
\end{equation}
\end{lemme}
\begin{proof}
Inequality $(4.49)$ in Antal--Pisztora~\cite{AP96} says that there exist an integer  $N$ and a real $c>0$ such that
\begin{equation*}
\label{apcache}
\forall \ell \ge 0\quad \P(0\communique y, D(0,y)>\ell)\le \P \left(\sum_{i=0}^n (|C_{i}|+1)>\ell c N^{-d}\right),
\end{equation*}
where the $C_{i}$'s are independent identically distributed random sets, such that there exists $h>0$ with $\E[\exp(h |C_{i}|)]<+\infty$, and $n$ is an integer
with $n\le \|y\|_1/2N$, which leads to $n+1\le \|y\|_1$.
Thus, for each $\ell>0$,
\begin{eqnarray*}
\label{apcachezz}
\P(0\communique y, D(0,y)\ge \ell) & \le  & \P(0\communique y, D(0,y)>\ell /2)\\
& \le & \P \left( \sum_{i=0}^n (|C_{i}|+1)>\frac \ell 2c N^{-d} \right)\\
& \le & \P \left( \sum_{i=1}^{\|y\|_1} \frac{2N^d}c (|C_{i}|+1)\ge \ell \right).
\end{eqnarray*}
Of course, the last inequality remains true for $\ell\le 0$, which proves that
$$\1_{\{0\communique y\}}D(0,y)\preceq \mu^{*\|y\|_1},$$
where $\mu$ is the distribution of 
$\frac{2N^d}c (|C_{0}|+1)$ and $\preceq$ denotes the stochastic domination.
Thus, we can choose  $\alphaa{APexp}=h\frac{c}{2N^d}$ and $\betaa{APexp}=\log \E [\exp(h |C_{0}|+1)]$.
\end{proof}

\begin{coro}
\label{danslecube}
There exist some  constants $\rho,\A{gecart},\B{gecart}>0$ such that
\begin{eqnarray}
\forall y\in\Zd \quad  \forall t\ge\rho \|y\|_1\quad  \P(0\communique y, \; D(0,y)>t)& \le &  \A{gecart}\exp(-\B{gecart} t). \label{gecart} 
\end{eqnarray}
\end{coro}

\begin{proof}
Let $y \in \Zd$ and $t>0$. With the Markov inequality and (\ref{APexp}), we get
$$
\P(0\communique y, \; D(0,y)>t)  \le e^{-\alphaa{APexp}t}\E[e^{\alphaa{APexp} \1_{ \{0\communique y \} }D(0,y)}] \le e^{\betaa{APexp}\|y\|_1-\alphaa{APexp}t}.
$$
To get (\ref{gecart}), we take  $\rho>2\frac{\betaa{APexp}}{\alphaa{APexp}}$.
%
%
\end{proof}

\subsection{Estimates for the modified chemical distance $D^*$}

We now show some estimates for $D^*$ that are analogous to the ones that we have obtained for the chemical distance $D$:
\begin{lemme}
\label{apstar}
There exist some constants $\rho^*, \A{apstare},\B{apstare},\alphaa{controlemomexp}, \betaa{controlemomexp},\A{controlnormedeux},\B{controlnormedeux}>0$ such that, for each $y\in\Zd$:

\begin{eqnarray}
\forall t\ge \rho^*\|y\|_1\quad \P(D^*(0,y)>t)&\le& \A{apstare}\exp(-\B{apstare} t), \label{apstare}\\ \E[e^{\alphaa{controlemomexp}D^*(0,y)}]& \le & e^{\betaa{controlemomexp}\|y\|_1}, \label{controlemomexp} \\
\label{controlnormedeux}
\|(D^*(0,y)-\rho^*\|y\|_1)^+\|_2& \le &\A{controlnormedeux}\exp(-\B{controlnormedeux}\|y\|_1).
\end{eqnarray}
\end{lemme}

\begin{proof}
Let $\lambda=\frac1{4\rho}$ and $\rho^*=2\rho$. Since $D^*(0,y)=D(0^*,y^*)$, we have 
\begin{eqnarray*}
\P(D^*(0,y)>t)& \le & \P(\|0^*\|_1 \ge \lambda t)+ \P( \|y^*-y\|_1 \ge \lambda t)\\
& &+\miniop{}{\sum}{\substack{ \|a\|_1\le \lambda t \\ \|b-y\|_1\le \lambda t}}  \P(\1_{\{a\communique b \}}D(a,b)>t).
\end{eqnarray*}
The first and  second terms are controlled with~(\ref{amasinfini}). If $t\ge \rho^* \|y\|_1$, then for each term in the sum, we have: 
$$
\|a-b\|_1\le \|y\|_1+2\lambda t\le  \frac{t}{2\rho} +2\lambda t=\frac{t}{\rho},
$$ 
which permits to use the bound in~(\ref{gecart}), and hence completes the proof of~(\ref{apstare}). 

Now, for $y\in \Zd\backslash\{0\}$, it follows from~(\ref{apstare}) that
\begin{eqnarray*}
\E [e^{\alpha D^*(0,y)}] & = & 1+\int_0^{+\infty} \P(D^*(0,y)>t)\alpha e^{\alpha t}dt \\
& \le & e^{\alpha \rho^* \|y\|_1} +\int_{\rho^* \|y\|_1}^{+\infty} \A{apstare}\exp(-\B{apstare} t)\alpha e^{\alpha t}dt,
\end{eqnarray*}
which gives~(\ref{controlemomexp}), taking for instance $\alphaa{controlemomexp}=\B{apstare}/2$.
Similarly, using~(\ref{apstare}) again,
\begin{eqnarray*}
\E([(D^*(0,y)-\rho^*\|y\|_1)^+]^2)  & = &  \int_0^{+\infty} \P(D^*(0,y)-\rho^*\|y\|_1>t) 2t \ dt \\
& \le &  \A{apstare} \exp(-\B{apstare}\|y\|_1 )\int_0^{+\infty} 2t\exp(-\B{apstare} t) \ dt,
\end{eqnarray*}
which gives~(\ref{controlnormedeux}).
\end{proof}

\subsection{Proof of theorem~\ref{concentrationD}}

Let us see how Theorem~\ref{concentrationD} follows from Theorem~\ref{concentrationD*}.
The estimate~(\ref{amasfini}) about the radius of large finite clusters makes us able to prove that on the event $\{0 \communique y\}$, the quantities $D(0,y)$ and $D^*(0,y)$ coincide with a huge probability: indeed, on the event $\{0\communique y, \; 0\communique \infty\}$, the identity $D(0,y)=D^*(0,y)$ holds, and we can bound the probability
$$\P(0\communique y, \; 0 \not\communique \infty) \le \A{amasfini}e^{-\B{amasfini} \|y\|_1}.$$

\begin{proof}[Proof of Theorem~\ref{concentrationD}]
We begin with the proof of~(\ref{concenD}). Let $y \in \Zd$.
\begin{eqnarray*}
& & \E (|D(0,y)-\mu(y)|\1_{\{0\communique y\}}) \\
 & = & 
\E (|D(0,y)-\mu(y)| \1_{\{0\communique y,\;0\communique\infty\}})
+ \E (|D(0,y)-\mu(y)|\1_{\{0\communique y,\;0\not\communique\infty\}}).
\end{eqnarray*}
On one hand, with the Cauchy--Schwarz inequality and the estimates~(\ref{gecart}) and~(\ref{amasfini}), we have
\begin{eqnarray*}
 && \E ( | D(0,y)-\mu(y) | \1_{\{0\communique y,0\not\communique\infty\}}) \\
&\le & \| (D(0,y)-\mu(y)) \1_{\{0\communique y\}} \|_2 \sqrt{\P(0\communique y,0\not\communique\infty)}\\
& \le &(\|D(0,y) \1_{\{0\communique y\}} \|_2+\mu(y)) \sqrt{\P(0\communique y,0\not\communique\infty)}\\
& =  & o \left( \sqrt{\|y\|_1}\log(1+\|y\|_1) \right).
\end{eqnarray*}
On the other hand, with the inequalities~(\ref{equetlesperance1}) and~(\ref{equvar}) in Theorem~\ref{concentrationD*},
\begin{eqnarray*}& & \E( | D(0,y)-\mu(y) | \1_{\{0\communique y,0\communique\infty\}})\\ & = &  \E( |D^*(0,y)-\mu(y)|\1_{\{0\communique y,0\communique\infty\}})\\
& \le & \E( | D^*(0,y)-\mu(y)| ) \\
& \le & |\E [ D^*(0,y)] -\mu(y)|+\sqrt{\Var(D^*(0,y))}\\
& \le & \C{equetlesperance1}\|y\|_1^{1/2}\log(1+\|y\|_1)+(\C{equvar}\|y\|_1\log(1+\|y\|_1))^{1/2},
\end{eqnarray*}
which proves~(\ref{concenD}).

Now we turn to the proof of~(\ref{equmoderter}). Let $y \in \Zd$. \\
For each $x\in [\C{equmoderbis} (1+\log \|y\|_1),\sqrt{\|y\|_1}]$, Inequality~(\ref{equmoderbis}) in Theorem~\ref{concentrationD*} and the estimate~(\ref{amasfini}) ensure  that
\begin{eqnarray*}
 && \P \left( \frac{| D(0,y)-\mu(y)|}{\sqrt{\|y\|_1}}> x , \; 0 \communique y\right) \\
 & \le & \P \left( \frac{| D^*(0,y)-\mu(y)|}{\sqrt{\|y\|_1}}> x \right) +\P(0 \communique y, \; 0 \not\communique +\infty) \\
& \le & \A{equmoderbis} e^{- \B{equmoderbis} x} +  \A{amasfini}e^{-\B{amasfini} \|y\|_1}.
\end{eqnarray*}
Since $x \le \sqrt{\|y\|_1} \le \|y\|_1$, this proves~(\ref{equmoderter}).

The proof of~(\ref{equtfa}) is standard from~(\ref{equmoderter}) -- see for instance the proof of Theorem 3.1 in Alexander~\cite{Alex97}.
\end{proof}

\section{Moderate deviations}

We prove here the concentration results~(\ref{equvar}) and~(\ref{equmoder}) for  $D^*$. 
The proof is of course based on Kesten's one~\cite{kesten-modere}; however, to overcome the lack of integrability of the chemical distance, we  use an approximation and renormalization process: for a real $t>0$ and any $k \in \Zd$, we denote by $\Lambda_k^t$ the set of edges whose centers are closer from $tk$ than from any other point of the grid $t\Zd$ (in equality cases, we use an arbitrary deterministic rule to associate the edge to a unique box). We say that a point in $\Zd$ is in the box $\Lambda_k^t$ if it is the extremity of an edge in this box: thus, the boxes $(\Lambda_k^t)_{k \in \Zd}$ partition the set of edges $\Ed$, but not the set of points $\Zd$.

We denote by $D^t(a,b)$ the distance obtained from the chemical distance as follows: if two points $x$ and $y$ are in the same box $\Lambda_k^t$, we add an extra red edge between them with length $Kt$, where $K$ is a constant such that $K>4\rho$.
For $t$ large enough, we will check with Lemma~\ref{pastropdiff} that $D^*(0,y)$ and $D^t(0^*,y^*)$ are really close.

Since $D^t(0,y)$ can by deterministically bounded, we can localize optimal paths in a deterministic box, which was not the case with $D^*$. Thus we can use a concentration result of Boucheron, Lugosi and  Massart~\cite{MR1989444} to bound the fluctuations of $D^t(0,y)$ around $\E[D^t(0,y)]$:
\begin{lemme}
\label{devDt}
For every $\C{devDgammax}>0$, there exist constants $\B{devDgammax},\gamma>0$ such that for every $y \in \Zd \backslash \{0\}$, 
\begin{equation}
\forall x \le  \C{devDgammax}\|y\|_1^{1/2} \quad  \P \left( \frac{|D^{\gamma x}(0,y)-\E [D^{\gamma x}(0,y)]|}{\sqrt{\|y\|_1}} \ge x \right)  \le  2\exp(- \B{devDgammax}x). \label{devDgammax}
\end{equation}
\end{lemme}

\begin{proof}
Let $\C{devDgammax}>0$, $y \in \Zd$ and $t>0$ be fixed. One of the ingredients of the proof is the existence of exponential moments for $D^t(0,y)$, which mainly follows from the existence of exponential moments for $D^*$. First, since we may use red edges,
\begin{equation}
\label{majDt}
D^t(0,y)\le Kt \left(\frac{\|y\|_1}t+1 \right)=K(\|y\|_1+t).
\end{equation}
Let us now see that there exist constants $\alpha, \beta, \eta>0$ such that
\begin{equation}
\label{momexpDt}
\forall y \in \Zd \backslash \{0\} \quad \forall t>0 \quad 
\log \E(\exp (\alpha D^t(0,y)) \le \beta \|y\|_1 +\eta t.
\end{equation}
Note first that the triangle inequality and~(\ref{majDt}) ensure that
\begin{eqnarray*}
D^t(0,y) 
& \le & D^t(0,0^*) + D^t(0^*,y^*) + D^t(y^*,y) \\
& \le & K( 2t + \|0-0^*\|_1 + \|y-y^*\|_1) + D^*(0,y).
\end{eqnarray*}
With H\"older's inequality, we get
$$\E(\exp (\alpha D^t(0,y)) \le \exp(2\alpha K t)[\E \exp(3 \alpha K\|0-0^*\|_1)]^{2/3}[\E \exp(3 \alpha D^*(0,y))]^{1/3}.$$
Inequalities~(\ref{amasinfini}) and~(\ref{controlemomexp}) give then the announced control~(\ref{momexpDt}).

By~(\ref{majDt}), any path realizing $D^t(0,y)$ stays inside a finite deterministic box: the quantity $D^t(0,y)$ only depends on the edges inside a finite family of mesoscopic boxes $(\Lambda_k^t)_{k \in \Zd}$, that we number from $1$ to $N$. Let 
$U_1,\dots, U_N$ be the random vectors such that $U_i$ contains the states of the edges in box number $i$. There exists a function
$S=S_{y,t}$ such that
$$D^t(0,y)=S(U_1,\dots,U_N).$$
Note that the $(U_i)$ are independent. Let $U'_1,\dots, U'_N$ be independent copies of $U_1,\dots,U_N$; set $S^{(i)}=S(U_1,\dots,U_{i-1},U'_i,U_{i+1},\dots, U_N)$ and 
\begin{eqnarray*}
V_{+} & = & \E \left[ \sum_{i=1}^N (( S-S^{(i)})_{+})^2| U_1,\dots,U_N \right], \\
V_{-} & = & \E \left[ \sum_{i=1}^N (( S-S^{(i)})_{-})^2| U_1,\dots,U_N \right].
\end{eqnarray*}
We can already note, with the Efron-Stein-Steele inequality (see Efron-Stein~\cite{MR615434} and Steele~\cite{MR840528} or Proposition 1 in Boucheron, Lugosi and Massart~\cite{MR1989444})  that
\begin{equation}
\Var D^t(0,y)\le\E[V_-].
\label{ESS}
\end{equation}
Moreover, Theorem~2 in Boucheron, Lugosi and Massart~\cite{MR1989444} gives the following concentration inequalities:
for every $\lambda, \theta>0$ such that $\lambda\theta<1$:
\begin{eqnarray}
\log\E[\exp(\lambda(S-\E[S]))] & \le & \frac{\lambda\theta}{1-\lambda\theta}\log\E \left[ \exp \left(\frac{\lambda V_+}{\theta}\right) \right], 
\label{BouLuMa+}\\
\log\E[\exp(-\lambda(S-\E[S]))] & \le & \frac{\lambda\theta}{1-\lambda\theta}\log\E \left[ \exp \left(\frac{\lambda V_-}{\theta}\right) \right].
\label{BouLuMa-}
\end{eqnarray}
Let $M^t(y,z)$ denote the shortest path for $D^t$ between $y$ and $z$, chosen with any deterministic rule in undetermined cases.  If we denote by $R_i$ the event "$M^t(0,y)$ crosses box number $i$", we can see that
$S^{(i)}-S\le Kt \1_{R_i}$. Thus, if $Y$ denotes the number of boxes visited by $M^t(0,y)$, then $K^2 t^2 Y$ is an upper bound for $V_-$.
Since $Y\le 3^d(1+D^t(0,y)/t)$, this leads to
\begin{equation}
V_-\le 3^dK^2 t(D^t(0,y)+t) \label{majV-}.
\end{equation}
This bound, together with the existence of exponential moments for 
$D^t(0,y)$, allows to bound the lower fluctuations. On the other hand,  we cannot obtain such a simple bound for $V_+$. Instead, we use a variant of Inequality~(\ref{BouLuMa+}), which was given to us by R. Rossignol and M. Théret:
\begin{lemme}
Assume there exist $\delta>0$, real functions $(\phi_i)_{1\le i\le n},(\psi_i)_{1\le i\le n}$  and $(g_i)_{1\le i\le n}$  such that for any $i$,  
$$(S-S^{(i)})_-\le\psi_i(U'_i)\text{ and }(S-S^{(i)})_-^2\le\phi_i(U'_i)g_i(U_1,\dots,U_n)$$
and $\alpha_i=\E[e^{\delta\psi_i(U_i)}\phi_i(U_i)]<+\infty$.

If 
$\displaystyle W=\sum_{i=1}^n \alpha_ig_i(U_1,\dots,U_n),$ then for every $\theta>0$ and every  $\lambda\in[0,\min(\delta,\frac1\theta))$, 
\begin{eqnarray}
\label{boulimatata}
\log\E[\exp(\lambda(S-\E[S]))] & \le & \frac{\lambda\theta}{1-\lambda\theta}\log\E \left[ \exp \left(\frac{\lambda W}{\theta}\right) \right].
\end{eqnarray}
\end{lemme}
Remember that $S^{(i)}-S\le Kt\1_{R_i}$; thus we can take $\phi_i=\psi_i=Kt$, $g_i=\1_{R_i}$ and $\alpha_i=Kt e^{\delta Kt}$; we get
$$W=K^2t^2 e^{\delta K t}Y \le 3^dK^2t e^{\delta K t}3^d(D^t(0,y)+t).$$
To recover an inequality similar to~(\ref{majV-}), we choose $\delta=1/t$:
\begin{equation}
\label{majW}
\max(V_-,W) \le 3^d K^2 e^{K} t(D^t(0,y)+t).
\end{equation}
Both sides can then be treated simultaneously, we only need to keep track of the two conditions: $\lambda <1/\theta$ and $\lambda<1/t$. 
Let us take the constants $\alpha,\beta,\eta$ given in~(\ref{momexpDt}), and choose
$\lambda,t>0$ such that
$$\left\{ 
\begin{array}{l}
\lambda<1/t, \\
\lambda^2\le \frac12 \frac{\alpha}{ 3^dK^2 e^{K}t}.
\end{array}
\right.$$
Setting
$\theta=\frac{\lambda 3^d K^2 e^{K} t}{\alpha}$, the condition $\theta\lambda\le 1/2$ is satisfied, so we can use~(\ref{BouLuMa-}) and~(\ref{boulimatata}). Since $\frac{\lambda}{\theta}3^d K^2e^{K} t=\alpha$, estimates~(\ref{majW}) and~(\ref{momexpDt}) say that
\begin{eqnarray} 
\frac{\lambda\theta}{1-\lambda\theta}\log\E\left[\exp\left(\frac{\lambda V_-}{\theta}\right)\right] 
&\le & 2\lambda\theta\log\E[\exp(\alpha(D^t(0,y)+t))] \nonumber \\
& \le & \frac{2.3^d K^2 e^{K}}{\alpha} \lambda^2t(\beta\|y\|_1+(\eta+\alpha)t) \nonumber \\
& \le & L\lambda^2 (t\|y\|_1+t^2) ,\label{majorL}
\end{eqnarray}
where $L$ is a constant such that $L \ge \frac{2.3^d K^2 e^{K}}{\alpha} \max(\beta,\eta+\alpha)$ and $L \ge \C{devDgammax}/2$ -- the last condition will be used at the very end of the proof. Similarly,
\begin{equation}
\frac{\lambda\theta}{1-\lambda\theta}\log\E\left[\exp\left(\frac{\lambda W}{\theta}\right)\right] \le L\lambda^2 (t\|y\|_1+t^2).
\label{majorelle}
\end{equation}
Thus, for any $u>0$  and any $\lambda,t>0$ such that $\lambda <1/t$ and $\lambda^2\le \frac12 \frac{\alpha}{3^dK^2 e^K t}$,  the Markov inequality,~(\ref{BouLuMa-})  and~(\ref{boulimatata}) ensure that
$$
\P \left( |D^t(0,y) -\E [D^t(0,y)]| >u \right)\le 2\exp(-\lambda u+L\lambda^2 (t\|y\|_1+t^2)).$$
But taking
$\displaystyle \lambda=\frac{x\|y\|_1^{1/2}}{2Lt(\|y\|_1+t)}<\frac{x}{2Lt\|y\|_1^{1/2}}$, we get
$$
\left\{ \begin{array}{l}
x \le 2L  \|y\|_1^{1/2}, \\
x^2 \le \frac{2\alpha}{3^d}\frac{L^2}{K^2 e^K}t\|y\|_1
\end{array} \right. 
\Rightarrow 
\left\{ \begin{array}{l}
\lambda<1/t, \\
\lambda^2\le \frac12 \frac{\alpha}{ 3^dK^2 e^{K}t}
\end{array} \right. 
$$
and taking now $u=x\sqrt{\|y\|_1}$, we obtain
$$\P \left( \frac{|D^t(0,y)-\E[D^t(0,y)]|}{\sqrt{\|y\|_1}}\ge x \right) \le 2\exp \left( -\frac{x^2\|y\|_1}{4L t(\|y\|_1+t)}\right).$$
At last, taking $t=\gamma x$ with $\gamma=\frac{3^d}{ \alpha} \frac{K^2 e^K}{L}$, we see that
$$ x\le  2L\|y\|_1^{1/2} \Rightarrow \left\{ \begin{array}{l}
\lambda<1/t, \\
\lambda^2\le \frac12 \frac{\alpha}{ 3^dK^2 e^{K}t}
\end{array} \right. 
$$
and we obtain, for any  $x\le \C{devDgammax} \|y\|_1^{1/2}\le  2L\|y\|_1^{1/2}$:
 $$\P \left( \frac{|D^{\gamma x}(0,y)-\E[ D^{\gamma x}(0,y) ] | }{\sqrt{\|y\|_1}}> x \right) \le 2\exp\left(-\frac{x}{4L\gamma  }\times \frac{1}{1+2L \gamma }\right),$$
 which ends the proof of Lemma~\ref{devDt}.
\end{proof}

\begin{lemme}
\label{pastropdiff}
There exist constants
$\A{presquepareil},\B{presquepareil},\A{esperance},\B{esperance}>0$ such that for every $y\in\Zd$ and every $t\le\|y\|_1$,
\begin{eqnarray}
\P(D^t(0^*,y^*)\ne D^*(0,y)) & \le & \A{presquepareil} (1+\|y\|_1)^{2d} \exp(-\B{presquepareil} t), \label{presquepareil} \\
\|D^*(0,y) -  D^t(0^*,y^*)\|_2 & \le & \A{esperance}(1+\|y\|_1)^{d+1}\exp(-\B{esperance} t).\label{esperance}
\end{eqnarray}
\end{lemme}

\begin{proof}
Set $\Gamma=\{x\in\Zd:\;\|x\|_1\le 3 \rho^*\|y\|_1)\}$ and
\begin{eqnarray*}
L & = & \{M^t(0^*,y^*)\subset \Gamma\}, \\
A  & = & \miniop{}{\cap}{a,b\in \Gamma:\; \|a-b\|_1\ge t}\{\1_{\{a\communique b\}}D(a,b)\le 2\rho\|a-b\|_1\}, \\
B & = & \miniop{}{\cap}{a\in \Gamma}\{a\in C_\infty\text{ or }C(a)\subset [-t,\dots,t]^d\}.
\end{eqnarray*}
Let us say that two boxes $\Lambda_k^t$ and $\Lambda_{\ell}^t$ are $*$-adjacent if $\|k-\ell\|_{\infty}=1$.
For $k\in\Zd$, we say that the box $\Lambda_k^t$ is good if any $x \in \Lambda_k^t$ communicating with any $y$ in $\Lambda_k^t$ or in one of the $3^d-1$ $*$-adjacent boxes is linked to $y$ by an open path with length smaller than $4\rho t$. We also set
$$G=\miniop{}{\cap}{\|k\|_1\le 1+3\rho^*\|y\|_1/t}\{\Lambda^t_k\text{ is good.}\}.$$ 
Let us prove that
\begin{equation}
\label{labelleinclusion}
L \cap A\cap B\cap  G \subset \{D^t(0^*,y^*)=D^*(0,y)\}.
\end{equation}
Let us focus on the event $L \cap A\cap B\cap  G$. On the event $L$, the optimal path $M^t(0^*,y^*)=(0^*=y_0,\dots, y_n=y^*)$ stays inside $\Gamma$. This path is composed by three types of portions: sequences of red edges, sequences of edges in the infinite cluster and sequences of edges in finite clusters. To jump from a finite cluster to the infinite cluster, the path has to use a red edge. 
Set
$$i_0=\max \{i: \;D(0^*,y_i)=D^t(0^*,y_i)\}.$$
To prove~(\ref{labelleinclusion}), we just need to show that $i_0=n$.
Since $D(0^*,y_{i_0})=D^t(0^*,y_{i_0})\le D^t(0^*,y^*)\le D^*(0,y)<+\infty$, the point $y_{i_0}$
is in the infinite cluster.

Assume by contradiction that $i_0<n$: the edge between $y_{i_0}$ and $y_{i_0+1}$ is not open, otherwise the maximality of $i_0$ would be denied. It is thus an extra red edge, with length $Kt$, added between two points of a same mesoscopic box: therefore, $D^t(y_{i_0},y_{i_0+1})=Kt$. If $y_{i_0}$ and $y_{i_0+1}$ are in the same open cluster, since $G$ occurs, $D(y_{i_0},y_{i_0+1})\le 4 \rho t \le Kt$. But then $D(y_{i_0},y_{i_0+1})\le D^t(y_{i_0},y_{i_0+1})$, which would contradict once again the maximality of $i_0$. Thus, 
$y_{i_0}$ and $y_{i_0+1}$ are not in the same open cluster, which means that
$y_{i_0+1}$ is not in the infinite cluster. Let
$$j_0=\inf\{j\in [i_0+1,\dots,n]: \; y_j\in C_{\infty}\}.$$ 
Since $y_n\in C_\infty$, we know that $j_0<+\infty$. Between $y_{i_0}$ and $y_{j_0}$, the path alternatively uses red edges and pieces of finite clusters. Look now at this microscopic path at the mesoscopic scale -- \ie consider the path of coordinates of mesocopic boxes successively visited. A site of the mesoscopic path is said to be red if the portion of the microscopic path crossing the corresponding mesoscopic box contains a red edge. Since event $B$ prevents any finite cluster in $\Gamma$ to link two non $*$-adjacent boxes, two consecutive red sites of the mesoscopic path can not be separated by more than one non-red site. Let us also remark that the edge linking $y_{i_0}$ (in the infinite cluster) and $y_{i_0+1}$ (in a finite cluster) has be be red. Thus at least half of the sites of the mesoscopic path are red. Consequently,
$$D(y_{i_0},y_{j_0})\ge D^t(y_{i_0},y_{j_0})\ge \frac{\|y_{i_0}-y_{j_0}\|_1}t \times\frac12\times  Kt=\frac{K}2\|y_{i_0}-y_{j_0}\|_1.$$
If $y_{i_0}$ and  $y_{j_0}$ are not in $*$-adjacent boxes, 
then $\|y_{i_0}-y_{j_0}\|_1 \ge t$ and at the same time, $D(y_{i_0},y_{j_0})\ge \frac{K}2\|y_{i_0}-y_{j_0}\|_1>2 \rho \|y_{i_0}-y_{j_0}\|_1$, which cannot occur on event $A$. Thus $y_{i_0}$ and $y_{j_0}$ are in $*$-adjacent boxes: the event $G$ ensures then that
$D(y_{i_0},y_{j_0})\le 4\rho t\le Kt\le D^t(y_{i_0},y_{j_0})$, which contradicts once again the maximality of $i_0$ and ends the proof of~(\ref{labelleinclusion}).

Thus
$\P(D^*(0,y)\ne D^t(0,y))\le \P(L^c)+\P(A^c)+\P(B^c)+\P(G^c)$, and it  now remains to bound each of these four probabilities.

To bound $\P(L^c)$, let us note that since $K\ge 1$, any point in $M^t(0^*,y^*)$ is at a $\|.\|_1$-distance less than $D^t(0^*,y^*)$ from $0^*$. Since $D^t(0^*,y^*)\le D^*(0,y)$, we have, using~(\ref{amasinfini}) and~(\ref{apstare}),
\begin{eqnarray*}
\P(L^c)& = & \P(M^t(0^*,y^*)\not\subset  \Gamma)\\
& \le & \P \left( \|0-0^*\|_1 \ge  \rho^* \|y\|_1 \right)  + \P \left( D^*(0,y)\ge 2\rho^*\|y\|_1 \right)\\
 & \le & \A{amasinfini}\exp(-\B{amasinfini} \|y\|_1)+\A{apstare}\exp(-2\rho^*\|y\|_1\B{apstare} \|y\|_1), 
\end{eqnarray*}
With~(\ref{gecart}),
\begin{eqnarray*}
\P(A^c) & = & \sum_{a,b\in \Gamma:\; \|a-b\|_1\ge t}\P( D(a,b)\ge 2\rho\|a-b\|_1) \\
& = & \sum_{a,b\in \Gamma:\; \|a-b\|_1\ge t} \A{gecart} \exp(-2\B{gecart} \rho\|a-b\|_1) \\
& = & (1+2\rho_2\|y\|_1)^{2d} \exp(-2\B{gecart} \rho t).
\end{eqnarray*}
With~(\ref{amasfini}),
$\P(B^c)  =  (1+3\rho^*\|y\|_1)^{d} \A{amasfini}\exp(-\B{amasfini} t).$\\
Remember that $\rho>2  \betaa{APexp} /\alphaa{APexp}$. Thus, if $\|\ell-k\|_\infty \le1$, $x \in \Lambda_k^t$ and $y \in \Lambda_\ell^t$, 
with~(\ref{gecart}),
\begin{eqnarray*}
\P( x \communique y, \; D(x,y) > 4\rho t ) & \le & \exp(-4 \alphaa{APexp} \rho t) \exp(\betaa{APexp} \|y-x\|_1) \\
& \le & \exp(-4 (\alphaa{APexp} \rho-2\betaa{APexp}) t).
\end{eqnarray*}
Thus, 
\begin{eqnarray*}
\P(G^c) & \le & \sum_{\|k\|_1\le 1+3\rho^*\|y\|_1/t} \P( \Lambda^t_k\text{ is not good}) \\
& \le & \sum_{\|k\|_1\le 1+3\rho^*\|y\|_1/t} \sum_{x \in \Lambda_k^t} \sum_{\|\ell-k\|_\infty \le1} \sum_{y \in \Lambda_\ell^t} \P( x \communique y, \; D(x,y) > 4\rho t )\\ 
& \le & (3+6\rho^*\|y\|_1/t)^d (2t+1)^d 3^d(2t+1)^d\exp(-4 (\alphaa{APexp} \rho-2\betaa{APexp}) t).
\end{eqnarray*}
Since $ t \le \|y\|_1$, we get~(\ref{presquepareil}).

For the second point, note that
$$
0\le D^*(0,y)-D^t(0^*,y^*)\le\rho^*\|y\|_1\1_{\{D^*(0,y)\ne D^t(0^*,y^*)\}}+(D^*(0,y)-\rho^*\|y\|_1)^+.
$$
Thus
\begin{eqnarray*}
& &\|D^*(0,y)-D^t(0^*,y^*)\|_2\\ & \le & \rho^*\|y\|_1\sqrt{\P(D^*(0,y)\ne D^t(0^*,y^*))}+\|(D^*(0,y)-\rho^*\|y\|_1)^+\|_2,
\end{eqnarray*}
and we conclude with~(\ref{presquepareil}) and~(\ref{controlnormedeux}), using once again that $\|y\|_1\ge t$.
\end{proof}

We can now move forward to the proof of~(\ref{equvar}) and~(\ref{equmoder}) in Theorem~\ref{concentrationD*}. The idea is quite simple: we use the estimates obtained for $D^t$ in Lemma~\ref{devDt}, and bound with Lemma~\ref{pastropdiff} the approximation error between $D^*$ and $D^t$.

\begin{proof}[Proof of~(\ref{equvar})]
We can write, for $y \in \Zd$,
\begin{eqnarray*}
& &\Var D^*(0,y)\\ & \le & 2(\Var D^t(0,y)+\Var(D^*(0,y)-D^t(0,y))\\
& \le &2\Var D^t(0,y)+4 \left( \E(D^*(0,y)-D^t(0^*,y^*))^2 +\E(D^t(0^*,y^*) -D^t(0,y))^2\right).
\end{eqnarray*}
We take $t=\frac{d+1}{\B{esperance}}\log(1+\|y\|_1)\le \|y\|_1$ as soon as $\|y\|_1$ is large enough.
\begin{itemize}
\item With~(\ref{controlnormedeux}),  we know that 
$\E[D^*(0,y)]=O(\|y\|_1)$; using~(\ref{ESS}) and~(\ref{majV-}), we get:
$$\Var D^t(0,y) \le 3^dK^2(t\E[D^*(0,y)]+t^2)=O(\|y\|_1 \log (1+\|y\|_1)).$$
\item Inequality~(\ref{esperance}) ensures that
\begin{eqnarray*}
\E(D^*(0,y)-D^t(0^*,y^*))^2 \le \A{esperance}^2 (1+\|y\|_1)^{2d+2}\exp(-2\B{esperance} t)=O(1).
\end{eqnarray*}
\item The triangle inequality for $D^t$ together with~(\ref{majDt}) ensure that
\begin{eqnarray*}
|D^t(0^*,y^*) -D^t(0,y)| & \le & D^t(0,0^*)+D^t(y,y^*) \\
& \le & K(\|0^*\|_1 +\|y-y^*\|_1 +2t).
\end{eqnarray*}
Minkowski's inequality and~(\ref{amasinfini}) say then that
$$\|D^t(0^*,y^*) -D^t(0,y)\|_2 =O(\log (1+\|y\|_1)),$$
\end{itemize}
which ends the proof of~(\ref{equvar}).
\end{proof}

\begin{proof}[Proof of~(\ref{equmoder})]

Let us remark first that it is sufficient to prove~(\ref{equmoder}) for $y$ large enough.
Let $\D{equmoder}>0$. Choose  $\gamma$ given by Lemma~\ref{devDt} with $\C{devDgammax}=\D{equmoder}$. 
We set
$$\C{equmoder}=\max\left\{1,\frac {4d}{\gamma\B{presquepareil}},\frac{d+1}{\gamma \B{esperance}}\right\}.$$
For every $y\in\Zd\backslash\{0\}$, if  $x\le \D{equmoder}\sqrt{\|y\|_1}$ and $x\ge \C{equmoder}(1+\log \|y\|_1)$,  then, using~(\ref{esperance}) and~(\ref{majDt}),
\begin{eqnarray*}
&& |\E[D^*(0,y)]-\E [D^{\gamma x}(0,y) ]| \\
& \le & |\E [D^*(0,y)]-\E [D^{\gamma x}(0^*,y^*) ]| + |\E[ D^{\gamma x}(0^*,y^*)]-\E[ D^{\gamma x}(0,y) ] | \\
& \le &\|D^*(0,y)-D^{\gamma x}(0^*,y^*)\|_1+\E[ D^{\gamma x}(0,0^*)]+\E[ D^{\gamma x}(y,y^*)] \\
& \le & \A{esperance} (1+\|y\|_1)^{d+1}\exp(-\B{esperance}\gamma \C{equmoder} (1+\log\|y\|_1))  \\
& & \quad \quad +K\E(\|0^*\|_1)+K\E(\|y-y^*\|_1)+2K\gamma x.
\end{eqnarray*}
With~(\ref{amasinfini}), we know that $\E(\|0^*\|_1)=\E(\|y-y^*\|_1)<+\infty$. Thus, since $\C{equmoder} \ge \frac{d+1}{\gamma \B{esperance}}$, for  $y$ large enough, if  $x\le\D{equmoder}\sqrt{ \|y\|_1}$ and $x\ge \C{equmoder}(1+\log \|y\|_1)$, we have
$$\frac{|\E [D^*(0,y)]-\E [D^{\gamma x}(0,y) ]|}{\sqrt{\|y\|_1}}  \le x/2,$$
which leads to
\begin{eqnarray*}
&& \P \left( \frac{| D^*(0,y)-\E [D^*(0,y)]|}{\sqrt{\|y\|_1}}> x \right) \\
& \le & \P(D^*(0,y)\ne D^{\gamma x}(0^*,y^*))
 +\P \left(\frac{| D^{\gamma x}(0^*,y^*)-\E [D^{\gamma x}(0,y)]|}{\sqrt{\|y\|_1}}> x/2\right).
\end{eqnarray*}
Since $x\ge \C{equmoder}(1+\log \|y\|_1)$ and $\C{equmoder}\ge \frac {4d}{\B{presquepareil}\gamma}$, estimate~(\ref{presquepareil}) ensures that
\begin{eqnarray*}
\P(D^*(0,y)\ne D^{\gamma x}(0^*,y^*)) & \le & \A{presquepareil} (1+\|y\|_1)^{2d} \exp(-\B{presquepareil} \gamma x) \\
& \le & \A{presquepareil}2^{2d} \exp \left( -\frac{\B{presquepareil}}2 \gamma x \right),
\end{eqnarray*}
For the second term, we write with~(\ref{majDt}):
\begin{eqnarray*}
&& | D^{\gamma x}(0^*,y^*)-\E [D^{\gamma x}(0,y)]| \\
& \le & | D^{\gamma x}(0,y)-\E[ D^{\gamma x}(0,y)]| + D^{\gamma x}(0,0^*) +D^{\gamma x}(y,y^*) \\
& \le & | D^{\gamma x}(0,y)-\E [D^{\gamma x}(0,y)]| + K \|0^*\|_1+ K \|y-y^*\|_1+ 2K\gamma x
\end{eqnarray*}
Then, for $y$ large enough,
\begin{eqnarray*}
&& \P \left(\frac{| D^{\gamma x}(0^*,y^*)-\E [D^{\gamma x}(0,y)] |}{\sqrt{\|y\|_1}}> x/2\right) \\
& \le & \P \left(\frac{| D^{\gamma x}(0,y)-\E [D^{\gamma x}(0,y)] |}{\sqrt{\|y\|_1}}> x/9\right)+ 2\P \left(\|0^*\|_1 \ge\frac{x\sqrt{\|y\|_1}}{9K} \right).
\end{eqnarray*}
Lemma~\ref{devDt} gives a bound of the correct order for the first term, while~(\ref{amasinfini}) gives a bound of the correct order for the second term, which ends the proof.
\end{proof}


\section{Asymptotic behavior of the mean value}

The aim of this section is to prove~(\ref{equetlesperance1}) in Theorem~\ref{concentrationD*}.

The function $D^*$ inherits the subadditivity from $D$. Besides, we can note that for each $a$ in $\Zd$, the joint distribution of $(D^*(x+a,y+a))_{x \in\Zd,y \in \Zd}$ does not depend on $a$.
Thus, if we define $h(x)=\E[D^*(0,x)]$, we have 
$$\forall x,y\in\Zd\quad h(x+y)\le h(x)+h(y).$$
Since  $h$ is subadditive, we can use the technics developped by Alexander~\cite{Alex97} for the approximation of subadditive functions. First, we prove the convergence of $h(ny)/n$ to $\mu(y)$:

\begin{lemme}
\label{limite}
For each $y\in\Zd$, the sequence $\frac{D^*(0,ny)}n$ converges almost surely and in $L^1$ to $\mu(y)$.
Particularly, $h(ny)/n$ converges to $\mu(y)$.
\end{lemme}

\begin{proof}
Since $D^*(0,y)$ is integrable and $(D^*(x,y))_{x \in\Zd,y \in \Zd}$ is stationary, the subadditive ergodic theorem tells us that there exists  $\mu^*(y)$ 
such that $D^*(ny)/n$ converges almost surely and in $L^1$ to $\mu^*(y)$. Thus, we just have to prove that $\mu^*(y)$ coincides with $\mu(y)$. 
By Garet--Marchand~\cite{GM-fpppc}, on the event $\{0\communique\infty\}$, we can almost surely find a sequence $(n_ky)_{k\ge 1}$ with $0\communique n_k y$ and $D(0,n_k y)/n_k$ tends to $\mu(y)$ when $k$ goes to infinity.
Obviously, $D^*(0,n_k y)/n_k$ converges to $\mu^*(y)$. Since the equality $D(0,n_k y)=D^*(0,n_k y)$ holds on  $\{0 \communique \infty\}$, we get that $\mu(y)=\mu^*(y)$.
\end{proof}

We recall the results of Alexander on the approximation of subadditive functions.
Let us introduce some notation derived from Alexander~\cite{Alex97}.
For some positive constants  $M$ and $C$, we define
$$
GAP(M,C)=\left\{
\begin{array}{c}
h:\Zd\to\R, \\
(\|x\|_1\ge M)\Rightarrow \left( \mu(x) \le h(x)\le \mu(x)+C\|x\|_1^{1/2}\log \|x\|_1 \right)
\end{array}
 \right\}.
$$
For $x \in \Rd$, we choose a linear form $\mu_x$ on $\Rd$ such that $\mu_x(x)=\mu(x)$ and such that
$$\forall y\in \mathcal B_{\mu}^0(\mu(x))\quad \mu_x(y)\le\mu(x).$$
The quantity $\mu_x(y)$ is the $\mu$-length of the projection of $y$ onto
the line from $0$ to $x$, following a hyperplane tangent to the convex set $\mathcal B_{\mu}^0(\mu(x))$ at  $x$.
It is easy to see that for each $y \in \Rd$, $|\mu_x(y)|\le \mu(y)$. Then, for
each positive constant $C$, we define 
$$
Q_x^h(C)=
\left\{
\begin{array}{c}
y\in\Zd:\|y\|_1\le (2d+1)\|x\|_1, \\
 \mu_x(y)\le\mu(x), \; h(y)\le\mu_x(y)+C\|x\|_1^{1/2}\log\|x\|_1
\end{array}
\right\}.
$$
The idea is that the elements of $Q_x^h(C)$ permit to realize a mesh of $\Zd$
with increments for which $\mu_x$ correctly approaches $h$. 
We also define, for $M>0,C>0,a>1$,
$$CHAP(M,C,a)=
\left\{
\begin{array}{c}
h:\Zd\to\R,\;(\|x\|_1\ge M)\Rightarrow
\left(
\begin{array}{c}
\exists \alpha \in[1,a], \\
x/\alpha\in \text{Co}( Q_x^h(C))
\end{array}
\right)
\end{array}
\right\},
$$
where $\text{Co}(A)$ is the convex hull of $A$ in $\Rd$. 
Alexander has the following results:

\begin{lemme}[Alexander~\cite{Alex97}]
\label{lemalex}
Let $h$ be a nonnegative subadditive function on $\Zd$, $M>1,C>0,a>1$ some fixed constants. We assume that for each $x\in\Zd$ with $\|x\|_1 \ge M$, there exist a natural number $n$, a lattice path $\gamma$ from $0$ to $nx$, and a sequence of points $0=v_0,v_1,\dots,v_m=nx$ in $\gamma$ such that $m\le an$ and whose increments $v_i-v_{i-1}$ belong to $Q_x^h(C)$. Then
$h\in CHAP(M,C,a)$
\end{lemme}

\begin{theorem}[Alexander~\cite{Alex97}]
\label{Alex2}
Let $h$ be a nonnegative subadditive function on $\Zd$ and  $M>1,C>0,a>1$ some fixed  constants. 
If $h\in CHAP(M,C,a)$, then $h\in GAP(M,C)$.
\end{theorem}

\begin{defi}
We call $Q_x^h(C)$-path every sequence $(v_0,\dots,v_m)$ such that for each $i \in \{0,\dots,m-1\}$, $v_{i+1}-v_i \in Q_x^h(C)$.

Let $\gamma=(\gamma(0),\dots, \gamma(n))$ be a simple lattice path in $\Zd$. 
We consider the unique sequence of indices $(u_i)_{0 \le i \le m}$ such that
$$
\begin{array}{l}
u_0=0, \; u_m=n, \\
\forall i \in \{0, \dots, m-1\} \quad \forall j \in \{u_i+1, \dots, u_{i+1}\} \quad \gamma(j)-\gamma(u_i) \in Q_x^h(C), \\
\forall i \in \{0, \dots, m-1\} \quad \gamma(u_{i+1}+1)-\gamma(u_i) \not \in Q_x^h(C).
\end{array}
$$
Then, the $Q_x^h(C)$-skeleton of $\gamma$ is the sequence $(\gamma(u_i))_{0 \le i \le m}$.
\end{defi}

We will prove the following result relative to the modified chemical distance $h(.)=\E [D^*(0,.)]$:

\begin{prop}
\label{propositionalex} There exist some constants $M>1$ and $C>0$  such that if $\|x \|_1 \ge M$, then for sufficiently large $n$ there exists a lattice path from $0$ to $nx$ with a $Q_x^h(C)$-skeleton of $2n+1$ or fewer vertices.
\end{prop}

Let us first see how this proposition leads to~(\ref{equetlesperance1}) and concludes the proof of Theorem~\ref{concentrationD*}.  

\begin{proof}[Proof of~(\ref{equetlesperance1})] 
Proposition~\ref{propositionalex} and Lemma~\ref{lemalex} ensure that $h(.)=\E [D^*(0,.)]$ is in $CHAP(M,C,2)$, which implies, thanks to Theorem~\ref{Alex2}, that $h$ is in $GAP(M,C)$. This gives~(\ref{equetlesperance1}) for each  $y \in \Zd$ such that $\|y\|_1 \ge M$, hence for each $y \in \Zd$, should we increase $\C{equetlesperance1}$.
\end{proof}

Let us go to the proof of Proposition~\ref{propositionalex}.
We now choose $h(.)=\E[D^*(0,.)]$, take $\beta$ and $C$ such that
\begin{equation}
\label{choixbetaC}
0<\beta<\B{equmoder}, \quad  C{}>\sqrt{2d}\left(\frac{d}{\beta}+\C{equmoder} \right)\text{ and } C'=48C{}.
\end{equation}
We define
\begin{eqnarray*}
{Q}_x & = & {Q}^h_x(C'), \\
G_x & = & \{y \in \Zd: \; \mu_x(y)>\mu(x)\}, \\
\Delta_x & = & \{y \in {Q}_x: \; y \text{ adjacent to } \Zd \backslash {Q}_x, \; y \text{ not adjacent to } G_x\}, \\
D_x & = & \{y \in {Q}_x: \; y \text{ adjacent to } G_x\}.
\end{eqnarray*}

\begin{lemme}
\label{accroissements}
There exists a constant $M$ such that if $\|x \|_1\ge M$, then
\begin{enumerate}
\item if $y \in Q_x$, then $\mu(y) \le 2 \mu(x)$ and $\|y\|_1 \le 2d \|x\|_1$;
\item if $y \in \Delta_x$, then $\E[D^*(0,y)]-\mu_x(y) \ge \frac{C'}2 \|x \|_1^{1/2}\log \|x\|_1$;
\item if $y \in D_x$, then $\mu_x(y) \ge \frac56 \mu(x)$;
\item if $x$ is large enough, $(\|y\|_1\le\|x\|_1^{1/2}) \Longrightarrow (y\in Q_x)$. 
\end{enumerate}
\end{lemme}

\begin{proof}
The arguments are simple and essentially deterministic. One can refer to Lemma~3.3 in Alexander~\cite{Alex97}, which is the analogous result in first-passage percolation. Particularly, we use the fact that $\E[D^*(0, \pm e_i)] \le \rho^*+A_{\ref{controlnormedeux}}<+\infty$, which plays the same role as the integrability of passage times does in first-passage percolation.
\end{proof}

We denote by $D(v_1,v_m;(v_i))$ the length of the shortest open path from  $v_1$ to $v_m$ that visits $v_1,\dots,v_m$ in this order. Then,
\begin{lemme}
\label{bovary}
$$\displaystyle \lim_{\|x\|_1 \to +\infty}\P \left(
\begin{array}{c}
\exists m \ge 1 \quad \exists \text{ a } Q_x\text{-path } (v_0=0, \dots, v_m): \\ 
\displaystyle \sum_{i=1}^{m-1} \E [D^*(v_i,v_{i+1})] - D(v_1,v_m;(v_i)) >C{} m \|x\|_1^{1/2} \log\|x\|_1
\end{array} 
\right)=0.$$
\end{lemme}

The proof of this lemma relies on a denumeration of the  $Q_x$-paths, on a bound on exponential moments for $\E [D^*(v_i,v_{i+1})]-D(v_i,v_{i+1})$ and on a BK-like inequality. Of course, it would be more natural to deal with quantities like $\E [D^*(v_i,v_{i+1})]-D^*(v_i,v_{i+1})$. 
But $D^*$ lacks the monotonicity property of $D$ needed to use this kind of tools.
Once more, we have to deal alternatively with $D$ and $D^*$.

\begin{proof}
Let us fix $m \ge1$ and $x$ large enough.
Let $(v_0=0, \; v_1, \; \dots, \; v_m)$ be a $Q_x$-path starting from $0$.
Lemma~\ref{accroissements} implies that for each $i$, $\|v_{i+1}-v_i\|_1 \le 2d \|x\|_1$. We define
\begin{eqnarray*}
Y_i & = & \E [D^*(v_i,v_{i+1})]-D^*(v_i,v_{i+1}) \\
Z_i & = & \E [D^*(v_i,v_{i+1})]-D(v_i,v_{i+1}).
\end{eqnarray*}
The moderate deviation result allows first to bound some exponential moments of $Y_i$ for large $x$. Indeed, we can write
$$\E \left[ \exp \left( \frac{\beta(Y_i)_+}{\sqrt{2d\|x\|_1}} \right) \right] =1+\int_{0}^{+\infty} \beta e^{\beta t} \P \left( {Y_i}\ge t \sqrt{2d\|x\|_1} \right)\ dt.$$
Let us note that~(\ref{controlnormedeux}) already gives the simple bound:
$$
\max(Y_i,Z_i)\le \E [D^*(v_i,v_{i+1})]\le \A{controlnormedeux}+\rho^* \|v_{i+1}-v_i\|_1;
$$
remarking that $\|v_{i+1}-v_i\|_1 \ge 1$, this particularly implies, for $\|x\|_1$ large enough, that
\begin{equation}
\label{chemisederegine}
\frac{\max(Y_i,Z_i)}{\sqrt{2d\|x\|_1}} 
\le 2 \rho^* \sqrt{\|v_{i+1}-v_i\|_1}.
\end{equation}
This ensures that $\P( {Y_i}\ge t\sqrt{2d\|x\|_1})=0$ as soon as $t>2 \rho^* \sqrt{\|v_{i+1}-v_i\|_1}$.
On the other hand, the moderate deviation result~(\ref{equmoder}) ensures the existence of constants $\A{equmoder},\B{equmoder}>0$ such that if $\C{equmoder}(1 +\log \|v_{i+1}-v_i\|_1) \le t \le 2 \rho^* \sqrt{\|v_{i+1}-v_i\|_1}$, then, since $\|v_{i+1}-v_i\|_1 \le 2d \|x\|_1$,
$$\P({Y_i}\ge t\sqrt{2d\|x\|_1}) \le \P({Y_i}\ge t\|v_{i+1}-v_i\|_1^{1/2})\le \A{equmoder} \exp(-\B{equmoder} t).$$
Thus, since $\beta<\B{equmoder}$, we get
\begin{eqnarray}
 \E \left[ \exp \left(  \frac{\beta(Y_i)_+}{\sqrt{2d\|x\|_1}} \right) \right]
& \le & 1+\int_0^{\C{equmoder}(1 +\log (2d\|x\|_1))}\beta e^{\beta t} dt + \int_0^{+\infty} \beta e^{\beta t} \A{equmoder} e^{-\B{equmoder} t} dt \nonumber \\
& \le & 1+ (2de)^{\beta \C{equmoder}} \|x\|_1^{\beta \C{equmoder}} +\frac{\A{equmoder} \beta}{\B{equmoder}-\beta}. 
\label{momoexp}
\end{eqnarray}

Let us note that the bound for the exponential moments of $Y_i$ is not as good as in first-passage percolation,
where it does not depend on  $\|x\|_1$ (see Alexander~\cite{Alex97}): this is due to the renormalization used to get the moderate deviations~(\ref{equmoder}).

We can remark that
\begin{itemize}
\item if $v_i \not \communique v_{i+1}$, then $(Z_i)_+=0$;
\item if $v_i \communique  v_{i+1}$ and $v_i \communique \infty$, then $(Z_i)_+=(Y_i)_+$;
\end{itemize}
Then, using~(\ref{chemisederegine}),
$$
\exp \left(  \frac{\beta(Z_i)_+}{\sqrt{2d\|x\|_1}} \right) 
 \le  1+\exp \left( \frac{\beta (Y_i)_+}{\sqrt{2d\|x\|_1}} \right)
 +
 \1_{ \left\{ \substack{v_i\communique v_{i+1} \\ v_i\not\communique\infty} \right\}}
 \exp \left( 2\beta\rho^* \sqrt{\|v_{i+1}-v_i\|_1} \right).
 $$
Using~(\ref{amasfini}) then~(\ref{momoexp}), we get, for $x$ large enough:
\begin{eqnarray}
& & \E \left[ \exp \left(  \frac{\beta(Z_i)_+}{\sqrt{2d\|x\|_1}} \right) \right] \label{mimiexp} \\
& \le & 1+ \E \left[ \exp \left(  \frac{\beta(Y_i)_+}{\sqrt{2d\|x\|_1}} \right) \right] +\exp \left( 2\beta\rho^* \sqrt{\|v_{i+1}-v_i\|_1} \right) \P \left(
\begin{array}{c}
v_i \communique v_{i+1}\\
v_i\not\communique \infty
\end{array}
\right)\nonumber  \\
& \le & 1+ \E \left[ \exp \left( \frac{\beta (Y_i)_+}{\sqrt{2d\|x\|_1}} \right) \right] +\A{amasfini} \exp \left( 2\beta\rho^* \sqrt{\|v_{i+1}-v_i\|_1}-\B{amasfini} \|v_{i+1}-v_i\|_1 \right) \nonumber\\
& \le &  (6d\|x\|_1)^{\beta\C{equmoder}}. \nonumber
\end{eqnarray}
We can apply a BK-like inequality to $D(v_1,v_m;(v_i))$: by Theorem~2.3 in Alexander~\cite{MR1202516}, for each $t>0$, if the $Z'_i$'s are independent copies of the  $Z_i$'s, then, with~(\ref{mimiexp}),
\begin{eqnarray*}
&& \P \left( \sum_{i=1}^{m-1} \E [D^*(v_i,v_{i+1})] - D(v_1,v_m;(v_i))>C{} m \|x\|_1^{1/2} \log\|x\| _1\right) \\
& \le & \P \left(\sum_{i=1}^{m-1} Z_i'>C{} m \|x\|_1^{1/2} \log\|x\|_1 \right) \\
& \le & \exp \left(-\frac{\beta C{} m \log\|x\|_1}{\sqrt{2d}} \right)\prod_{i=1}^{m-1}\E\left[\exp \left( \frac{ \beta Z_i}{\sqrt{2d\|x\|_1}} \right)\right]\\
& \le & \left( (6d)^{\beta\C{equmoder}} \|x\|_1^{-\frac{\beta C{}}{\sqrt{2d}} + \beta\C{equmoder}} \right) ^m.
\end{eqnarray*} 
From Lemma~\ref{accroissements}, there exists a constant $K$ such that there are at most $(K\|x\|^d)^m$ $Q_x$-pathss of $m+1$ points starting from $0$. Then, summing on these paths, we get
\begin{eqnarray*}
 &  & \P \left(
\begin{array}{c}
\exists \text{ a } Q_x\text{-path } (v_0=0, \dots, v_m): \\ 
\displaystyle \sum_{i=1}^{m-1}  \E [D^*(v_i,v_{i+1})] - D(v_1,v_m;(v_i)) >C{} m \|x\|_1^{1/2} \log\|x\|_1
\end{array} 
\right) \\
& \le & \left( K(6d)^{\beta\C{equmoder}} \|x\|_1^{d-\frac{\beta C{}}{\sqrt{2d}} + \beta\C{equmoder}} \right) ^m = \left( \frac{L}{\|x\|_1^{\alpha}} \right)^m,
\end{eqnarray*}
where $L$ and $\alpha$ are some constants;  the choice~(\ref{choixbetaC}) we did for $\beta$ and $C{}$ ensures that $\alpha>0$. Thus,
for $x$ large enough, $L \|x\|_1^{-\alpha}\le \frac12 $, so summing over all possible lengths~$m$:
$$
\P \left(
\begin{array}{c}
\exists m \ge 1 \quad \exists \text{ a } Q_x \text{-path } (v_0, \dots, v_{m}): \\    
\displaystyle \sum_{i=1}^{m-1}  \E [D^*(v_i,v_{i+1})] - D(v_1,v_m;(v_i))  >C{} m \|x\|_1^{1/2} \log\|x\|_1
\end{array} 
\right) \\
 \le  \frac{2 L}{ \|x\|_1^{\alpha}},
$$
which completes the proof of the lemma.
\end{proof}

\begin{proof}[Proof of Proposition~\ref{propositionalex}]
This proposition corresponds to Proposition~3.4. in Alexander~\cite{Alex97}. 

The proof of this deterministic result uses the so-called ``probabilistic method'': we are going to prove that with positive probability, we can construct such a $Q_x$-path  from  a path realizing  $D^*(0,nx)$.
With Lemma~\ref{bovary} and~(\ref{amasinfini}), we can find $M>1$ such that if $\|x\|_1\ge M$, then 
\begin{eqnarray*}
 \P\left(
\begin{array}{c}
\exists m \ge 1 \quad \exists \text{ a } Q_x\text{-path } (v_0=0, \dots, v_m): \\ 
\displaystyle \sum_{i=1}^{m-1}  \E [D^*(v_i,v_{i+1}) ]- D(v_1,v_m;(v_i)) >C{} m \|x\|_1^{1/2} \log\|x\|_1
\end{array} 
\right) & \le & \frac15, \\
\text{et } \P \left(\|0^*\| \ge \|x\|_1^{1/2} \right) & \le & \frac15.
\end{eqnarray*}
Should we take a greater  $M$, we can also assume that if $\|y\|_1\le \|x\|_1^{1/2}$ and $\|x\|_1\ge M$, then $y \in Q_x$.
Now fix $x\in \Zd$ with $\|x\|_1\ge M$. 
Lemma~\ref{limite} ensures that there exists $n_0 \in \N$ such that
$$\forall n \ge n_0 \quad \P(D^*(0,nx)> n(\mu(x)+1))\le \frac15.$$
Let  $n \ge n_0$. With probability at least $1/5$, the following properties hold:
\begin{enumerate}[a)]
\item for each $m \ge 1$, for each $Q_x$-path $(v_0=0, \dots, v_m)$,
$$\sum_{i=1}^{m-1} \E [D^*(v_i,v_{i+1})] - D(v_1,v_{m};(v_i))\le C{} m \|x\|_1^{1/2} \log\|x\|_1,$$
\item $D^*(0,nx) \le n(\mu(x)+1)$,
\item $\|0^*\| _1\le \|x\|_1^{1/2}$,
\item $\|(nx)^*-nx\|_1 \le \|x\|_1^{1/2}$.
\end{enumerate}
Then, we can find $\omega$ fulfilling these properties, and work from now on with this particular realization. Let $v_1=0^*,\dots, v_{m}=(nx)^*$ be the $Q_x$-skeleton of some path~$\gamma$ realizing
the chemical distance from $0^*$ to $(nx)^*$. Define $v_0=0$ and $v_{m+1}=nx$: properties c) and d) ensure that we get a $Q_x$-path from $0$ to $nx$. We will prove that $m+2 \le 2n+1$, which will conclude the proof of the proposition. Note that, by construction, 
 $$D(v_1,v_{m};(v_i))=D^*(0,nx).$$
From a), our $Q_x$-path satisfies, for this particular $\omega$,
\begin{equation}
\label{numero}
  \sum_{i=1}^{m-1} \E[D^*(v_i,v_{i+1})] - D(v_1,v_{m};(v_i))\le C{} m \| x\|_1^{1/2} \log\|  x\|_1.
\end{equation}
Besides, using the definition of $Q_x$, property d),  the fact that $\mu$ is a norm and~(\ref{gecart}), we get
\begin{eqnarray*}
m\mu(x)&\ge& \sum_{i=0}^{m-1}\mu_x(v_{i+1}-v_i)=\mu_x((nx)^*)\\
& =& \mu_x((nx))+\mu_x((nx)^*-nx)\\
& \ge& n\mu(x)-\mu((nx)^*-nx)
 \ge   n\mu(x)-\rho\|x\|_1^{1/2}.
\end{eqnarray*}
Thus, should we take a larger $M$, we get  $n\le \frac{11}{10}m$. 
Now, with b) and taking once more a larger $M$ if necessary,
\begin{eqnarray*} 
\sum_{i=1}^{m-1} \E [D^*(v_i,v_{i+1})] 
& \le & D^*(0,nx)+ C{} m \| x\|_1^{1/2} \log\| x\|_1\\
&\le  &   n(\mu(x)+1)+C{} m \| x\|_1^{1/2} \log\| x\|_1\\
& \le & n\mu(x)+2C{} m \| x\|_1^{1/2} \log\| x\|_1.
\end{eqnarray*}
We now distinguish in the $Q_x$-skeleton the short increments and the long ones:
\begin{eqnarray*}
S((v_i)) & = & \{i: \; 1 \le i \le m-1, \; v_{i+1}-v_i \in \Delta_x\}, \\
L((v_i)) & = & \{i: \; 1 \le i \le m-1, \; v_{i+1}-v_i \in D_x\}.
\end{eqnarray*}
By the definition of a $Q_x$-skeleton, these two sets partition $\{1,\dots,m-1\}$. 
Let us first bound the number of short increments with Lemma~\ref{accroissements}, estimate~(2): recall that $\mu_x(y)\le \mu(y)\le \E[D^*(0,y)]$, so
\begin{eqnarray*}
\sum_{i=1}^{m-1}\E [D^*(v_i,v_{i+1})] & = & \sum_{i=1}^{m-1}[ \mu_x(v_{i+1}-v_i)+(\E D^*(v_i,v_{i+1})-\mu_x(v_{i+1}-v_i))]\\
& \ge & \mu_x((nx)^*)-\mu_x(0^*)+|S((v_i))| \frac{C'}2 \|x\|_1^{1/2}\log\|x\|_1\\
& \ge & n\mu(x)-2\rho\|x\|_1^{1/2}+|S((v_i))| \frac{C'}2 \|x\|_1^{1/2}\log\|x\|_1.
\end{eqnarray*}
Thus, summing the two estimates and enlarging $M$ if necessary, we get
\begin{eqnarray*}
|S((v_i))| \frac{C'}2 \|x\|_1^{1/2}\log\|x\|_1
& \le & 2\rho\|x\|_1^{1/2} +2C{} m \| x\|_1^{1/2} \log\| x\|_1 \\
& \le  & 3C{} m \| x\|_1^{1/2} \log\| x\|_1,
\end{eqnarray*}
hence
$$|S((v_i))|\le 6\frac{C{}}{C'}m=\frac{m}{8}.$$
Similarly, we bound the number of long increments with Lemma~\ref{accroissements}, estimate~(3): 
\begin{eqnarray*}
\sum_{i=1}^{m-1}\E[ D^*(v_i,v_{i+1})]&=&\sum_{i=1}^{m-1}[ \mu_x(v_{i+1}-v_i)+(\E [D^*(v_i,v_{i+1})]-\mu_x(v_{i+1}-v_i))]\\
& \ge & n\mu(x)-2\rho\|x\|_1^{1/2}+\frac56 |L((v_i))|\mu(x) \ge \frac56 |L((v_i))|\mu(x),
\end{eqnarray*}
increasing $M$ if necessary; this gives
$$|L((v_i))|\le\frac{6}{5}n+2C{} m \frac{\| x\|_1^{1/2} \log\| x\|_1}{\mu(x)}\le\frac{6}{5}n+m/8.$$
Finally,
$m=|S((v_i))|+|L((v_i))|+\le \frac{6}{5}n+m/4$,
so $m\le\frac{8}{5}n$, which concludes the proof.
\end{proof}

We would like to thank Raphaël Rossignol and Marie Théret for  kindly giving us the extension of Boucheron, Lugosi and Massart's result.


\def\refname{References}
\bibliographystyle{plain}
\bibliography{mdev-en}

\begin{thebibliography}{10}

\bibitem{MR1202516}
Kenneth~S. Alexander.
\newblock A note on some rates of convergence in first-passage percolation.
\newblock {\em Ann. Appl. Probab.}, 3(1):81--90, 1993.

\bibitem{Alex97}
Kenneth~S. Alexander.
\newblock Approximation of subadditive functions and convergence rates in
  limiting-shape results.
\newblock {\em Ann. Probab.}, 25(1):30--55, 1997.

\bibitem{AP96}
Peter Antal and Agoston Pisztora.
\newblock On the chemical distance for supercritical {B}ernoulli percolation.
\newblock {\em Ann. Probab.}, 24(2):1036--1048, 1996.

\bibitem{MR2451057}
Michel Bena{\"{\i}}m and Rapha{\"e}l Rossignol.
\newblock Exponential concentration for first passage percolation through
  modified {P}oincar\'e inequalities.
\newblock {\em Ann. Inst. Henri Poincar\'e Probab. Stat.}, 44(3):544--573,
  2008.

\bibitem{MR1989444}
St{\'e}phane Boucheron, G{\'a}bor Lugosi, and Pascal Massart.
\newblock Concentration inequalities using the entropy method.
\newblock {\em Ann. Probab.}, 31(3):1583--1614, 2003.

\bibitem{CCGKS}
J.~T. Chayes, L.~Chayes, G.~R. Grimmett, H.~Kesten, and R.~H. Schonmann.
\newblock The correlation length for the high-density phase of {B}ernoulli
  percolation.
\newblock {\em Ann. Probab.}, 17(4):1277--1302, 1989.

\bibitem{DS}
R.~Durrett and R.~H. Schonmann.
\newblock Large deviations for the contact process and two-dimensional
  percolation.
\newblock {\em Probab. Theory Related Fields}, 77(4):583--603, 1988.

\bibitem{MR615434}
B.~Efron and C.~Stein.
\newblock The jackknife estimate of variance.
\newblock {\em Ann. Statist.}, 9(3):586--596, 1981.

\bibitem{GM-fpppc}
Olivier Garet and R{\'e}gine Marchand.
\newblock Asymptotic shape for the chemical distance and first-passage
  percolation on the infinite {B}ernoulli cluster.
\newblock {\em ESAIM Probab. Stat.}, 8:169--199 (electronic), 2004.

\bibitem{GM-large}
Olivier Garet and R{\'e}gine Marchand.
\newblock Large deviations for the chemical distance in supercritical
  {B}ernoulli percolation.
\newblock {\em Ann. Probab.}, 35(3):833--866, 2007.

\bibitem{Grimmett-Marstrand}
G.~R. Grimmett and J.~M. Marstrand.
\newblock The supercritical phase of percolation is well behaved.
\newblock {\em Proc. Roy. Soc. London Ser. A}, 430(1879):439--457, 1990.

\bibitem{MR2023652}
C.~Douglas Howard.
\newblock Models of first-passage percolation.
\newblock In {\em Probability on discrete structures}, volume 110 of {\em
  Encyclopaedia Math. Sci.}, pages 125--173. Springer, Berlin, 2004.

\bibitem{MR1452554}
C.~Douglas Howard and Charles~M. Newman.
\newblock Euclidean models of first-passage percolation.
\newblock {\em Probab. Theory Related Fields}, 108(2):153--170, 1997.

\bibitem{MR1849171}
C.~Douglas Howard and Charles~M. Newman.
\newblock Geodesics and spanning trees for {E}uclidean first-passage
  percolation.
\newblock {\em Ann. Probab.}, 29(2):577--623, 2001.

\bibitem{kesten-modere}
Harry Kesten.
\newblock On the speed of convergence in first-passage percolation.
\newblock {\em Ann. Appl. Probab.}, 3(2):296--338, 1993.

\bibitem{pimentel-preprint}
L.~Pimentel.
\newblock Asymptotics for first-passage times on {D}elaunay triangulations.
\newblock {\em preprint, available at
  \verb+http://arxiv.org/abs/math.PR/0510605+}, 2005.

\bibitem{MR840528}
J.~Michael Steele.
\newblock An {E}fron-{S}tein inequality for nonsymmetric statistics.
\newblock {\em Ann. Statist.}, 14(2):753--758, 1986.

\bibitem{MR1361756}
Michel Talagrand.
\newblock Concentration of measure and isoperimetric inequalities in product
  spaces.
\newblock {\em Inst. Hautes \'Etudes Sci. Publ. Math.}, (81):73--205, 1995.

\end{thebibliography}


\end{document}